\documentclass[10pt]{article}

\usepackage{amsthm,amsmath,amssymb}
\usepackage{yfonts}

\usepackage{ytableau}
\usepackage{graphicx,url}

\usepackage{tikz-cd}
\usetikzlibrary{decorations.pathreplacing}

\usepackage{ytableau}
\usepackage[colorlinks=true,citecolor=blue,linkcolor=blue,urlcolor=blue]{hyperref}

\newcommand{\arxiv}[1]{\href{http://arxiv.org/abs/#1}{\texttt{arXiv:#1}}}

\theoremstyle{plain}
\newtheorem{theorem}{Theorem}[section]
\newtheorem{lemma}[theorem]{Lemma}
\newtheorem{corollary}[theorem]{Corollary}

\theoremstyle{definition}
\newtheorem{definition}[theorem]{Definition}
\newtheorem{example}[theorem]{Example}
\newtheorem{conjecture}[theorem]{Conjecture}

\theoremstyle{remark}

\title{\bf Linear function of a poset}

\author{Stefan Mitrovi\'c\footnote{Corresponding author\\ Address : stefan.mitrovic@matf.bg.ac.rs}\\
\small Faculty of Mathematics\\[-0.8ex]
\small University of Belgrade\\[-0.8ex]
\small Serbia\\
\small\tt stefan.mitrovic@matf.bg.ac.rs\\
}

\begin{document}

\maketitle

\begin{abstract}

Stanley and Grinberg introduced a symmetric function associated with digraphs and named it the Redei-Berge symmetric function \cite{S}. This function arises from a suitable combinatorial Hopf algebra on digraphs \cite{GS}, which made it possible to assign the Redei-Berge function to posets \cite{MS}. In this paper, we define a new combinatorial Hopf algebra of posets whose character is a close cousin of the character from \cite{MS}. Further, we investigate the properties of the symmetric function that arises from this algebra and explore its expansions in various natural bases of $QSym$ and $Sym$. Finally, we obtain an interesting method for decomposing a poset. 


\end{abstract}

\section{Introduction}

Over the years, many Hopf algebras and symmetric functions of posets have emerged. Some of them can be found in \cite{JR}, \cite{WRS}, \cite{E}, \cite{ABM}, \cite{BS} and \cite{MS}. This concept gives birth to a global algebraic framework for studying partially ordered sets.


The Redei-Berge symmetric function is a recently introduced symmetric function assigned to digraphs \cite{S}. Its expansion in the power sum basis easily yields two beautiful theorems from combinatorics - Redei's theorem \cite{LR} and Berge's theorem \cite{CB}. In \cite{GS}, the authors proved that this function arises from a suitable combinatorial Hopf algebra of digraphs. This reconceptualization made it possible to define the Redei-Berge function for posets by defining a combinatorial Hopf algebra of posets compatible with the one on digraphs \cite{MS}. Its character counts the number of quasi-linear extensions of a poset, which is a starting point for the research in this paper.

In this paper, we introduce one new symmetric function associated with posets. This function is induced by an appropriate combinatorial Hopf algebra whose character counts the number of linear extensions of a poset. It turns out that this character is connected to the Euler character of a generalization of the well-known Rota's combinatorial Hopf algebra from \cite{JR} and \cite{ABS}. This connection  yields a special way to decompose a poset. Furthermore, we explore the combinatorial interpretation of the coefficients in the expansion of this function in many natural bases of $Sym$ and $QSym$. Ultimately, in the last section, we propose some conjectures that are yet to be answered.

\section{Preliminaries}

In this section, we review some basic notions and facts. A {\it set partition} $\pi=\{V_1,\ldots,V_{l(\pi)}\}\vdash V$ of length $l(\pi)$
of a finite set $V$ is a set of disjoint nonempty subsets with $V_1\cup\ldots\cup V_{l(\pi)}=V$. A {\it set composition} $\sigma=(V_1,\ldots,V_{l(\sigma)})\models V$ is an ordered partition. 

A \textit{composition} $\alpha\models n$ of length $l(\alpha)$, where $n\in\mathbb{N}$, is a sequence $\alpha=(\alpha_1,\ldots,\alpha_{l(\alpha)})$ of
positive integers with $\alpha_1+\cdots+\alpha_{l(\alpha)}=n$. We say that $n$ is the weight of $\alpha$ and write $n=|\alpha|$. The set of all compositions of $n$ will be denoted by $\mathrm{Comp}(n)$.
The \textit{type of a set composition} $\sigma=(V_1,\ldots,V_{l(\sigma)})\models V$ is the composition $\mathrm{type}(\sigma)=(|V_1|,\ldots,|V_{l(\sigma)}|)\models |V|$. 
A \textit{partition} $\lambda\vdash n$ is a
composition $(\lambda_1,\ldots,\lambda_{l(\lambda)})\models n$ such that $\lambda_1\geq \lambda_2\geq\cdots\geq
\lambda_{l(\lambda)}$. We denote by $P(n)$ the set of all partitions of $n$. The \textit{diagram} of a composition $\alpha=(\alpha_1, \ldots, \alpha_{l(\alpha)})$ is a collection of cells justified on the left such that the $i$th row contains $\alpha_i$ cells. We enumerate the rows starting from the top.

There is a bijection  $\mathrm{set}: \mathrm{Comp}(n)\rightarrow2^{[n-1]}$ given by
$\mathrm{set}: (\alpha_1,\ldots,\alpha_{l(\alpha)})\mapsto\{\alpha_1,\alpha_1+\alpha_2,\ldots, \alpha_1+\cdots+\alpha_{l(\alpha)-1}\}$. The inverse of this bijection is denoted by $I\mapsto\mathrm{comp}(I)$. We say that composition $\alpha\models n$ is \textit{finer} than composition  $\beta\models n$, and write $\alpha\leq\beta$, if $\mathrm{set}(\beta)\subseteq\mathrm{set}(\alpha)$. Equivalently, parts of $\beta$ are obtained by summing some adjacent parts of $\alpha$.

If $V$ is a set with $n$ elements, a $V-$\textit{listing} is a list of all elements of $V$ with no repetitions, i.e. a bijective map $\omega:[n]\rightarrow V$. We write $\Sigma_V$ for the set of all $V$-listings. A \textit{reversion} of a $V$-listing $\omega=(\omega_1,\ldots,\omega_n)\in\Sigma_V$ is a $V$-listing $\omega^\mathrm{rev}=(\omega_n,\ldots,\omega_1)$. For a subset $I\subseteq[n-1]$ we denote by $I^{\mathrm{op}}=\{n-i|i\in I\}$ its \textit{opposite set}. If $\alpha\models n$, we write $\alpha^{\mathrm{op}}$ for a composition of $n$ whose corresponding set is $\mathrm{set}(\alpha)^{\mathrm{op}}$. In other words, if $\alpha=(\alpha_1, \ldots, \alpha_{l(\alpha)})$, then $\alpha^{\mathrm{op}}=(\alpha_{l(\alpha)}, \ldots, \alpha_1)$.

Throughout this paper, $\mathbf{k}$ will denote a field of characteristic 0. A {\it combinatorial Hopf algebra}, CHA
for short, $(\mathcal{H},\zeta)$  over a field $\mathbf{k}$ is a graded, connected Hopf
algebra $\mathcal{H}=\oplus_{n\geq 0}\mathcal{H}_n$ over $\mathbf{k}$ together with
a $\mathbf{k}-$algebra homomorphism
$\zeta:\mathcal{H}\rightarrow\mathbf{k}$ called the {\it character}. The theory of combinatorial Hopf algebras is founded in \cite{ABS}.

The \textit{convolution product} $\zeta_1\zeta_2$ of two characters $\zeta_1, \zeta_2: \mathcal{H}\rightarrow\mathbf{k}$ is given by \[
\mathcal{H}\xrightarrow{\Delta_{\mathcal{H}}}\mathcal{H}\otimes\mathcal{H}\xrightarrow[]{\zeta_1\otimes\zeta_2}\mathbf{k}\otimes\mathbf{k}\xrightarrow[]{m_\mathbf{k}} \mathbf{k},
\] where $\Delta_\mathcal{H}$ is the coproduct of $\mathcal{H}$ and $m_\mathbf{k}$ is the product of the base field. For a character $\zeta$, its \textit{conjugate character} $\overline{\zeta}$ is defined as $\overline{\zeta}(h)=(-1)^n\zeta(h)$ if $h\in\mathcal{H}_n$ is a homogeneous element of degree $n$. To $\zeta$, we also assign its \textit{Euler character} $\chi=\zeta\overline{\zeta}$ and its \textit{Möbius} character $\zeta^{-1}=\zeta\circ S$, where $S$ is the antipode of $\mathcal{H}$. If $\varphi, \psi$ are two characters on $\mathcal{H}$, with $S(\varphi, \psi)$ will be denoted the largest graded subcoalgebra such that $\forall h\in S(\varphi, \psi)$, $\varphi(h)=\psi(h)$. In particular, $S_+(\mathcal{H}, \zeta)=S(\overline{\zeta}, \zeta)$ and $S_-(\mathcal{H}, \zeta)=S(\overline{\zeta}, \zeta^{-1})$ are respectively called \textit{even} and \textit{odd subalgebra} of $(\mathcal{H}, \zeta)$. 

The terminal object in the category of CHAs is the
CHA of quasisymmetric functions $(QSym,\zeta_Q)$. For basics of
quasisymmetric functions see \cite{EC} and for the Hopf algebra structure on $QSym$ see \cite{ABS}. A composition $\alpha=(\alpha_1,\ldots,\alpha_k)\models n$ defines the \textit{monomial quasisymmetric function}

\[M_\alpha=\sum_{i_1<\cdots<
i_k}x_{i_1}^{\alpha_1}\cdots x_{i_k}^{\alpha_k}.\]  Alternatively, we write $M_I=M_{\mathrm{comp}(I)},\ I\subseteq[n-1]$. The collection of all monomial quasisymmetric functions forms a linear basis of $QSym$. Another basis of this space consists of the \textit{fundamental quasisymmetric functions }

\[
F_I=\sum_{\substack{1\leq i_1\leq i_2\leq\cdots\leq i_n\\
                  i_j<i_{j+1} \ \mathrm{for \ each} \ j\in I}}x_{i_1}x_{i_2}\cdots x_{i_n}, \quad I\subseteq[n-1],
\]
which are expressed in the monomial basis by
\[
F_I=\sum_{I\subseteq J}M_J, \quad I\subseteq[n-1].
\]

The unique canonical morphism
$\Psi:(\mathcal{H},\zeta)\rightarrow(QSym,\zeta_Q)$ for a CHA $(\mathcal{H}, \zeta)$ is given on
homogeneous elements with

\begin{equation}\label{canonical}
\Psi(h)=\sum_{\alpha\models n}\zeta_\alpha(h)M_\alpha, \
h\in\mathcal{H}_n,
\end{equation} where $\zeta_\alpha$ is the convolution product
\begin{equation}\label{generalcoeff}
\zeta_\alpha=\zeta_{\alpha_1}\cdots\zeta_{\alpha_k}:\mathcal{H}
\stackrel{\Delta^{(k-1)}}\longrightarrow\mathcal{H}^{\otimes
k}\stackrel{proj}\longrightarrow\mathcal{H}_{\alpha_1}\otimes\cdots\otimes\mathcal{H}_{\alpha_k}
\stackrel{\zeta^{\otimes k}}\longrightarrow\mathbf{k}.
\end{equation}

The \textit{principal specialization} $\mathrm{ps}^1:QSym\rightarrow\mathbf{k}[m]$ is an algebra homomorphism defined by
\[\mathrm{ps}^1(\Phi)(m)=\Phi(\underbrace{1,\ldots,1}_{\text{$m$ ones}},0,0,\ldots).\]

The \textit{reciprocity formula} expresses values of the principal specialization at negative integers
\begin{equation}\label{reciprocity}
\text{ps}^1(\Phi)(-m)=\text{ps}^1(S(\Phi))(m),
\end{equation} where $S$ is the antipode of $QSym$.

The CHA of symmetric functions $(Sym, \zeta_S)$ is the terminal object in the category of
cocommutative combinatorial Hopf algebras \cite{ABS}. This algebra consists of quasisymmetric functions that are invariant under the action of permutations on the set of variables. The theory of symmetric functions can be found in \cite{EC} and the details of the Hopf algebra structure on $Sym$ in \cite{ABS}. There are many bases of this vector space and the basis of our particular interest consists of monomial symmetric functions. For partition $\lambda$, \textit{monomial symmetric function} $m_\lambda$ is defined as \[m_\lambda=\sum_{\alpha\sim\lambda}M_\alpha,\] where the summation runs over all compositions $\alpha$ whose parts, when arranged in nonincreasing order, yield $\lambda$.
The second basis we will take advantage of consists of elementary symmetric functions. The $i$th \textit{elementary symmetric function} is given by 
\[e_0=1\hspace{5mm} \textrm{ and } \hspace{5mm}e_i=\sum_{j_1<j_2<\cdots<j_i}x_{j_1}x_{j_2}\cdots x_{j_i} \] for $i\geq 1$.
For partition $\lambda=(\lambda_1, \lambda_2, \ldots, \lambda_k)$, we define
\[e_{\lambda}=e_{\lambda_1}e_{\lambda_2}\cdots e_{\lambda_k}.\]
Another basis of our interest is 
the \textit{power sum basis}. The $i$th \textit{power sum symmetric function} is defined as
\[p_0=1 \hspace{5mm} \text{and} \hspace{5mm}p_i=\sum_{j=1}^{\infty}x_j^i\]
for $i\geq 1$. For partition $\lambda=(\lambda_1, \lambda_2, \ldots, \lambda_k)$, we define 

\[p_{\lambda}=p_{\lambda_1}p_{\lambda_2}\cdots p_{\lambda_k}.\] 
 On the space $Sym$, there is a well-known automorphism $\omega: Sym\rightarrow Sym$, defined on this basis by $\omega(p_n)=(-1)^{n-1}p_n$. Consequently, $\omega(p_\lambda)=(-1)^{|\lambda|-l(\lambda)}p_\lambda$. For partition $\lambda$, we define the \textit{complete symmetric function} $h_\lambda$ as \[h_\lambda=\omega(e_\lambda).\] The collection of all complete symmetric functions also forms a basis of $Sym$.
The final basis we will be interested in consists of \textit{Schur functions}. For $\lambda=(\lambda_1, \ldots, \lambda_k)$, we define the \textit{Schur function} $s_\lambda$ as \[s_\lambda=\mathrm{det}(h_{\lambda_i-i+j})_{i, j=1}^k.\] 

The space of symmetric functions can be equipped with a scalar product $\langle, \rangle$ defined with  $\langle m_\lambda, h_\mu\rangle=\delta_{\lambda, \mu}$ for every two partitions $\lambda, \mu$. Relative to this scalar product, Schur functions form an orthonormal basis, i.e. $\langle s_\lambda, s_\mu\rangle=\delta_{\lambda, \mu}.$

If $u=\{u_i\}$ is a basis of some vector space $V$, $v\in V$ is $u-$\textit{positive} if all the coefficients in the expansion $v=\sum c_i u_i$ are nonnegative. If these coefficients are integers, we say that $v$ is $u-$\textit{integral}. Additionally, we write $[u_i]v$ for the coefficient $c_i$ of $u_i$ in $u-$expansion of $v$.

A \textit{graph} $G$ is a pair $G=(V, E)$, where $V$ is a finite set and $E$ is a collection $E\subseteq \binom{V}{2}$ of unordered pairs of elements of $V$. Elements of $V$ are called \textit{vertices} and elements of $E$ are called \textit{edges}. 
We say that $G=(V, E)$ is \textit{discrete} if $E=\emptyset$ and that it is a \textit{clique} if $E=\binom{V}{2}$.

A {\it digraph} $X$ is a pair $X=(V,E)$, where $V$ is a finite set and $E$ is a collection
$E\subseteq V^2$ of ordered pairs of elements of $V$. Elements of $V$ are the \textit{vertices} and
elements of $ E$ are the \textit{edges} of $X$.

 Let $\mathcal{D}=\oplus_{n\geq0}\mathcal{D}_n$ be the graded vector space over the field $\mathbf{k}$, where each $\mathcal{D}_n$ is linearly spanned by the set of all digraphs with $[n]$ as the set of vertices. For digraphs $X=(V,E)\in \mathcal{D}_m$ and $Y=(V',E')\in\mathcal{D}_n$, the product $X\cdot Y$ is the digraph with the vertex set $[m+n]$ and with the set of directed edges defined as follows:
\begin{itemize}
\item  for $x, y\in[m]$, $(x, y)$ is an edge of $X\cdot Y$ if and only if $(x, y)\in E$
\item  for $x, y \in [n]$,  $(x+m, y+m)$ is an edge of $X\cdot Y$  if and only if $(x, y)\in E'$
\item for $x\in [m]$ and $y\in \{m+1, \ldots, m+n\}$, $(x, y)$ is an edge of $X\cdot Y$. 
\end{itemize}
 
 The \textit{restriction} of a digraph $X=(V,E)$ on a subset $S\subseteq V$ is the
digraph $X|_S=(S,E|_S)$, where $E|_S=\{(u,v)\in E\  |\ u,v\in S\}$. The graded comultiplication $\Delta:\mathcal{D}\rightarrow\mathcal{D}\otimes\mathcal{D}$ is defined as

\[\Delta(X)=\sum_{W\subseteq V}X|_{W}\otimes X|_{V\setminus W}.\] Evidently, this is a coassociative and a cocommutative operation. The unit element is $\emptyset$, while the counit is given by $\epsilon(\emptyset)=1$ and $\epsilon(X)=0$ otherwise. It is not difficult to see that with these operations, $\mathcal{P}$ becomes a graded connected Hopf algebra \cite{ABM}.

Let $P$ be a finite poset.  The relation that induces $P$ will be denoted as $\leq_P$ and the corresponding strict order relation as $<_P$. 
For a $P$-listing $\omega=(\omega_1, \ldots, \omega_n)$, we say that it is a \textit{linear extension} of $P$ if there do not exist $i<j$ such that $\omega_{j}\leq_P\omega_i$. A \textit{quasi-linear extension} of $P$ is defined in \cite{MS} as a $P$-listing $(\omega_1, \ldots, \omega_n)$  such that $\omega_{i+1}\nleq_P \omega_i$ for all $i\in[n-1]$. Clearly, every linear extension is also a quasi-inear extension. We denote by $P^*$ the \textit{dual poset} of $P$, where $a\leq_{P^*}b$ if and only if $b\leq_P a$. The \textit{incomparability graph} of $P$, denoted as $\mathrm{inc}(P)$, is a graph with $P$ as the set of vertices, while $\{i, j\}$ is an edge of $\mathrm{inc}(P)$ if and only if $i$ and $j$ are incomparable in $P$.

Let $\mathcal{P}_n$ denote the $\textbf{k}$-span of all partial orders on $[n]$. The graded vector space $\mathcal{P}=\bigoplus_{n\geq 0}\mathcal{P}_n$ can become a graded Hopf algebra as follows. The graded multiplication $\mathcal{P}_m\otimes\mathcal{P}_n\rightarrow\mathcal{P}_{m+n}$ is the linear extension of the operation given by
 $$P\otimes Q\mapsto PQ \textrm{ for } P \textrm{ a poset on } [m] \textrm{ and $Q$ a poset on } [n]$$
where $PQ$ denotes the poset on $[m+n]$ with $\leq_{PQ}$ defined as follows: 
\begin{itemize}
\item  for $x, y\in[m]$, $x\leq_{PQ}y$ if and only if $x\leq_P y$
\item  for $x, y \in [n]$,  $x+m\leq_{PQ} y+m$ if and only if $x\leq_Q y$
\item for $x\in [m]$ and $y\in \{m+1, \ldots, m+n\}$, $x\leq_{PQ} y$. 
\end{itemize}
The poset $PQ$ defined above is known as the \textit{ordinal sum} of posets $P$ and $Q$. The unit is obviously the map $u:\mathbf{k}\rightarrow \mathcal{P}$ given by $1\mapsto \emptyset$, where $\emptyset$ is the unique poset on the empty set. A poset $P$ is \textit{irreducible} if it cannot be written as $P=P_1P_2$ for some two nonempty posets $P_1$ and $P_2$.

The \textit{standardization} of a finite set $A\subseteq\mathbb{N}$ is the unique order preserving map $\text{st}: A\rightarrow\{1, \ldots, |A|\}$. For a poset $P$ on some finite set $A\subseteq\mathbb{N}$, we define the \textit{standardization} of $P$, denoted as $\text{st}(P)$, as a poset on $\{1, \ldots, |A|\}$ obtained from $P$ via the standardization of $A$. The graded comultiplication on $\mathcal{P}$ is given by 
$$P\mapsto \sum_{S\subseteq [n]} \text{st}(P|_{S})\otimes \text{st}(P|_{[n]\setminus S}),$$ and linearly extended on 
$\mathcal{P}_{n}$. We define the counit by $\epsilon(\emptyset)=1$ and $\epsilon(P)=0$ otherwise. With operations defined in this way, $\mathcal{P}$ becomes a graded connected Hopf algebra that is cocommutative, but not commutative, see \cite{ABM}.

To any poset $P$, we can assign a digraph $D_P$ as follows. The set of vertices of $D_P$ is $P$, while $(i, j)$ is an edge of $D_P$ if and only if $i<_P j$. Moreover, it turns out that this assignment gives rise to an injective morphism of graded Hopf algebras $\mathcal{P}$ and $\mathcal{D}$ \cite{ABM}.


\section{Generalization of Rota's Hopf algebra}

In this section, we slightly modify the construction of the well-known Rota's combinatorial Hopf algebra of posets \cite{JR}, \cite{ABS}. This construction will be very useful in the upcoming section. The terms we are about to introduce are originally defined only for posets with unique minimal and maximal element in which every maximal chain has the same length (the so-called \textit{chain condition}). However, we are going to define these concepts in a more general setting.

Let $P$ be a finite poset. \textit{Height} of such poset $P$, denoted as $h(P)$ is the length of the longest chain in $P$. For $p\in P$, we define $H_P(p)=\{q\in P\mid q\leq_Pp\}$ and $D_P(p)=\{q\in P\mid p\leq q \}$. The \textit{height} $h(p)$ of an element $p\in P$ is $h(p)=h(H_P(p))$, while its \textit{depth} $d(p)$ is $d(p)=h(D_P(p))$.

We say that a poset $P$ is a \textit{mountain} if it has a unique maximal element denoted as $1_P$. For a mountain $P$ of height $h$, its $\widetilde{f}-$\textit{vector} $\widetilde{f}(P)$ is the list $(f_0, \ldots, f_h)$, where $f_i$ is the number of elements of depth $i$. Clearly, if $P$ is a mountain, then $h(P)=h(1_P)$ and $f_0=1$.  

Let $\widetilde{\mathcal{R}}$ be a vector space over the field $\mathbf{k}$, whose basis is the set of all classes
of isomorphic finite mountains. $\widetilde{\mathcal{R}}$ is a graded algebra, where grading is given by the height of a poset and the product of two posets $P_1$ and $P_2$ is their Cartesian product $P_1\times P_2$. Clearly, the unit is a poset with one element, which we denote as $\{\star\}$.
We can define comultiplication on $\widetilde{\mathcal{R}}$ by \[\Delta(P)=\sum_{p\in P}H_P(p)\otimes D_P(p).\] Note that this map is coassociative and compatible with multiplication, which means that  $\Delta(P_1\times P_2)=\Delta(P_1)\times\Delta(P_2)$. This follows from the fact that $H_{P_1\times P_2}(p_1, p_2)=H_{P_1}(p_1)\times H_{P_2}(p_2)$ and $D_{P_1\times P_2}(p_1, p_2)=D_{P_1}(p_1)\times D_{P_2}(p_2)$. Unfortunately, comultiplication does not respect grading as shown in the following example. However, \[
    \Delta(\widetilde{\mathcal{R}}_n)\subseteq\bigoplus_{i+j\leq n}\widetilde{\mathcal{R}}_i\otimes\widetilde{\mathcal{R}}_j\]

\begin{example}
 Let $P$ be the poset whose digraph $D_P$ is shown in the diagram below.
    
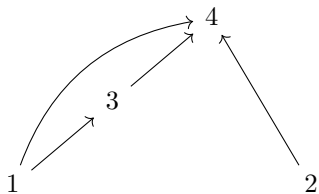
\begin{center}
   \begin{tikzcd}
                                     &              & 4 &               \\
                                     & 3 \arrow[ru] &   &               \\
1 \arrow[ru] \arrow[rruu, bend left] &              &   & 2 \arrow[luu]
\end{tikzcd}
    \end{center}
   $\Delta(P)=\{1\}\otimes\{1, 3, 4\}+\{2\}\otimes\{2, 4\}+\{1, 3\}\otimes\{3, 4\}+\{1, 2, 3, 4\}\otimes\{4\}$. However, $h(P)=2\neq 0+1=h(\{2\})+h(\{2, 4\})$.
\end{example}

Obviously, if we proposed the chain condition on our posets, we would avoid this problem. However, even then the adequate counit could not be defined since we do not have a unique minimal element. Therefore, we do not even have a coalgebra structure on $\widetilde{R}$. Additionally, the posets of our interest do not satisfy the chain condition.

Finally, we introduce the character $\zeta$ with $\zeta(P)=1$ for every poset $P$. Since $\widetilde{\mathcal{R}}$ is a graded algebra, it makes sense to define $\overline{\zeta}$ and $\chi=\zeta\overline{\zeta}$ in the same way as in the theory of combinatorial Hopf algebras. Clearly, \[\chi(P)=\zeta\overline{\zeta}(P)=\sum_{p\in P}(-1)^{d(p)}=\sum_{i=0}^{h(P)}(-1)^{i}f_i.\]

\subsection{Sets of border points}

In the sequel, we give an example of the construction we have made. This particular example will be of great value in the following section. 

A nonempty collection $\mathcal{A}$ of subsets of $[n-1]$ is a \textit{plucking} if $S\in\mathcal{A}$ and $S\subseteq T$ imply $T\in\mathcal{A}$. In other words, $\mathcal{A}$ is an upper set in $2^{[n-1]}$. Clearly, the collection of minimal sets in any plucking $\mathcal{A}$ uniquely determines $\mathcal{A}$. Note that any plucking contains $[n-1]$, the unique maximal element. Therefore, any plucking is a mountain. Hence, it makes sense 
to observe the value of the Euler character $\chi(\mathcal{A})$. Moreover, if $S\in\mathcal{A}$, then clearly $H_\mathcal{A}(S)=\{T\in \mathcal{A}\mid T\subseteq S\}$ is also a mountain with the unique maximal element $S$. In what follows, we write $\chi_{\mathcal{A}}(S)$ for $\chi(H_{\mathcal{A}}(S))$ and $\chi_\mathcal{A}$ for $\chi(\mathcal{A}).$ Note that $\chi_\mathcal{A}=\chi_\mathcal{A}([n-1]).$
\begin{example}
    Let $\mathcal{A}=2^{[n-1]}$. Then, for any nonempty $S\in\mathcal{A}$, $\chi_\mathcal{A}(S)=\sum_{T\subseteq S}(-1)^{|S|-|T|}=(1-1)^{|S|}=0.$ Clearly, for $S=\emptyset$, $\chi_{\mathcal{A}}(S)=1.$
\end{example}


\begin{lemma} \label{nestajanje}
    Let $\mathcal{A}$ be a plucking with  minimal sets $M_1, \ldots, M_k$ and $S\in\mathcal{A}$. If there exists $i\in S\setminus\bigcup_{i=1}^{k}M_i$, then $\chi_\mathcal{A}(S)=0$.
\end{lemma}

\begin{proof}
Let $T\in\mathcal{A}$ such that $T\subseteq S$. If $i\notin T$, then $T\cup\{i\}\in\mathcal{A}$ and $T\cup\{i\}\subseteq S$. On the other hand, if $i\in T$, then obviously $T\setminus\{i\}\subseteq S$. Since $T\in\mathcal{A}$, there exists $M_j$ such that $M_j\subseteq T$. Since $i\notin M_j$, $M_j\subseteq T\setminus \{i\}$ and therefore, $T\setminus\{i\}\in \mathcal{A}$ because $\mathcal{A}$ is a plucking. Hence, we can define an involution $f$ on $H_\mathcal{A}(S)=\{T\in \mathcal{A}\mid T\subseteq S\},$ with \[f(T) =
  \begin{cases}
  T\cup\{i\} & \text{ if } i\notin T \\
  T\setminus\{i\} & \text{ if } i\in T
  \end{cases}\]Since $T$ and $f(T)$ contribute to \[\chi_{\mathcal{A}}(S)=\sum_{\substack{T\subseteq S\\
                  T\in\mathcal{A}}}(-1)^{|S|-|T|}\] with opposite signs, they cancel each other out, which completes the proof.
\end{proof}

Let $\mathcal{A}$ be a plucking with minimal sets $M_1, \ldots, M_k$. The previous lemma tells us that the only sets $S\in\mathcal{A}$ that survive the action of $\chi_{\mathcal{A}}$ are the ones that are contained in $\bigcup_{i=1}^{k}M_i$. We say that a plucking is \textit{complete} if  $\bigcup_{i=1}^k M_i=[n-1]$. Note that for a plucking $\mathcal{A}$ that is not complete, $\chi_{\mathcal{A}}=0.$ However, there are complete pluckings $\mathcal{A}$ such that $\chi_\mathcal{A}=0$.

\begin{example} \label{kontraprimer}
    Let $\mathcal{A}=\{\{1, 2\}, \{1, 3\}, \{2, 4\}, \{1, 2, 3\}, \{1, 2, 4\}, \{1, 3, 4\},$
    \\$\{2, 3, 4\}, \{1, 2, 3, 4\}\}$. Then, $\mathcal{A}$ is a complete plucking, but $\chi_{\mathcal{A}}=1-1-1-1-1+1+1+1=0$.
\end{example}

In what follows, we give an example of a plucking that naturally emerges from the combinatorial structure of posets. This construction captures the compatibility of a listing $\omega\in\Sigma_P$ and the relation $\leq_P$ of a given poset $P$.

If $\omega\in\Sigma_P$, we say that a composition $\alpha=(\alpha_1, \ldots, \alpha_k)\models n$ is \textit{compatible} with $\omega$ if $(\omega_1, \ldots, \omega_{\alpha_1}), \ldots, (\omega_{\alpha_1+\cdots+\alpha_{k-1}+1}, \ldots, \omega_n)$ are linear extensions of adequate subposets. The set of all compositions that are compatible with $\omega$ will be denoted by $\mathrm{Comp}(P, \omega)$. For $\alpha\in \mathrm{Comp}(P, \omega)$, we call $\mathrm{set}(\alpha)$ \textit{set of border points} of $\omega$.  Clearly, if $\alpha\in \mathrm{Comp}(P, \omega)$ and $\beta\leq\alpha$, then $\beta\in\mathrm{Comp}(P, \omega)$ also.  In other words, the collection of all sets of border points of $\omega$ is a plucking. By abuse of notation, this collection is also denoted as $\mathrm{Comp}(P, \omega)$. Likewise, the set of all listings $\omega\in\Sigma_P$ such that composition $\alpha$ is compatible with $\omega$ is $\Sigma_P(\alpha)$, or equivalently $\Sigma_P(\mathrm{set}(\alpha))$.

\begin{example}
    If $\omega\in\Sigma_P$ is a linear extension of an $n$ element poset $P$, then $\mathrm{Comp}(P, \omega)=2^{[n-1]}$.
\end{example}

Since any $\mathrm{Comp}(P, \omega)$ is a plucking, we will be interested in finding minimal sets of border points of $\omega$, or equivalently, maximal compositions in $\mathrm{Comp}(P, \omega)$. Furthermore, we may calculate $\chi_{\mathrm{Comp}(P, \omega)}$ for any $\omega$. Motivated by Lemma \ref{nestajanje}, we will say that a $P-$listing $\omega$ is $P$-\textit{reversing} if $\mathrm{Comp}(P, \omega)$ is a complete plucking. This means that for every $i\in[n-1]$ there exists a minimal set of border points that contains $i$, which implies that for every $i\in[n-1]$, there exist $j\leq i<k$ such that $\omega_k\leq_P\omega_j$. In particular, $\omega_1$ is not a minimal element of $P$ and $\omega_n$ is not a maximal element of $P$. The previous example tells us that, in some sense, linear extensions are the furthest from being $P-$reversing listings since the only minimal set of border points for them is $\emptyset$. Note that $\chi_{\mathrm{Comp}(P, \omega)}\neq 0$, only if $\omega$ is $P-$reversing. The set of all $P-$reversing listings will be denoted as $\mathrm{Rev}(P)$.

\begin{example}
   Let $P$ be the poset whose digraph $D_P$ is shown in the diagram below.
    
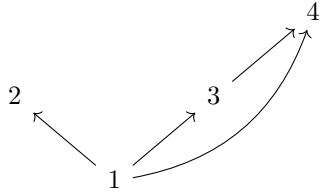
\begin{center}
   \begin{tikzcd}
  &                                                  &              & 4 \\
2 &                                                  & 3 \arrow[ru] &   \\
  & 1 \arrow[lu] \arrow[ru] \arrow[rruu, bend right] &              &  
\end{tikzcd}
    \end{center}
    If $\omega=(1, 4, 2, 3)$, maximal compositions in $\mathrm{Comp}(P, \omega)$ are $(3, 1)$ and $(2, 2)$. Equivalently, minimal sets of border points are $\{3\}$ and $\{2\}$. Therefore, this listing is not $P-$reversing (which could be seen immediately since $\omega_1=1$ is a minimal element of $P$).
\end{example}

If $P$ is not a chain, then $\mathrm{Comp}(P, \omega)$ is never a one-element plucking since we can get $[n-1]$ as the minimal set of border points of $\omega$ if and only if $\omega_1\geq_P\omega_2\geq_P\cdots\geq_P\omega_n$. One may be deceived by thinking that all the minimal sets from $\mathrm{Comp}(P, \omega)$ need to have the same cardinality. However, this is not true in general, as shown in the following example.

\begin{example}
    Let $P$ be the poset whose digraph $D_P$ is shown in the diagram below.
    
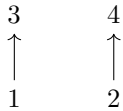
\begin{center}
    \begin{tikzcd}
3           & 4           \\
1 \arrow[u] & 2 \arrow[u]
\end{tikzcd}
    \end{center}
    If $\omega=(4, 3, 2, 1)$, maximal compositions in $\mathrm{Comp}(P, \omega)$ are $(2, 2)$ and $(1, 2, 1)$. Therefore, minimal sets of border points are $\{2\}$ and $\{1, 3\}$ and hence, $\omega\in\mathrm{Rev}(P)$.
\end{example}

Clearly, if $i_1, \ldots, i_k$ are border points of $\omega$ in $P$, then $n-i_k, \ldots, n-i_1$ are border points of $\omega^{\mathrm{rev}}$ in $P^*$. In other words, $\alpha\in\mathrm{Comp}(P, \omega)$ if and only if $\alpha^{\mathrm{op}}\in\mathrm{Comp}(P^*, \omega^{\mathrm{rev}})$ and vice versa. For a collection $\mathcal{A}$ of some compositions of $n$, $\mathcal{A}^{\mathrm{op}}=\{\alpha^{\mathrm{op}}\mid \alpha\in\mathcal{A}\}$ and for a collection $\mathcal{B}$ of $P-$listings, $\mathcal{B}^{\mathrm{rev}}=\{\omega^{\mathrm{rev}}\mid \omega\in\mathcal{B}\}.$ With all this in mind, we obtain the following.

\begin{lemma} \label{lemica}
    For any poset $P$ and $\omega\in\Sigma_P$, $\mathrm{Comp}(P, \omega)=\mathrm{Comp}(P^{*}, \omega^{\mathrm{rev}})^{\mathrm{op}}.$ Equivalently, for any poset $P$ on $[n]$ and any $S\subseteq[n-1]$, $\Sigma_P(S)=\Sigma_{P^*}(S^{\mathrm{op}})^{\mathrm{rev}}.$
\end{lemma}

\section{Linear function of a poset}

As we have already noted, the Redei-Berge function of a poset $P$ arises from a CHA of posets with the character that counts the number of quasi-linear extensions of $P$. In what follows, we investigate the properties of one function that can be seen as a cousin of the Redei-Berge function for reasons that will soon become clear.

Hopf algebra of posets $\mathcal{P}$ can be endowed with yet another character $\zeta$ as follows. For a poset $P$, let $\zeta(P)$ be the number of linear extensions of $P$. Clearly, if $(\omega_1, \ldots, \omega_n)$ is a linear extension of $P_1P_2$, then $\omega_1, \ldots, \omega_{|P_1|}\in P_1$ and $\omega_{|P_1|+1}, \ldots, \omega_n\in P_2$. Moreover, $(\omega_1, \ldots, \omega_{|P_1|})$ is a linear extension of $P_1$ and $(\omega_{|P_1|+1}, \ldots, \omega_n)$ is a linear extension of $P_2$. Therefore, $\zeta$ is multiplicative.

\begin{definition}
    Let $\Psi: (\mathcal{P}, \zeta)\rightarrow (QSym, \zeta_Q)$ be the unique canonical homomorphism of CHAs. We define the \textit{linear function} $L_P$ of a poset $P$ as $L_P=\Psi(P)$.
\end{definition}

The definition immediately implies that this function is multiplicative and that it does not distinguish $P$ from $P^*$ since $\zeta(P)=\zeta(P^*)$.

\begin{corollary}
    If $P_1$ and $P_2$ are posets, \[L_{P_1P_2}=L_{P_1}L_{P_2}.\]
\end{corollary}

\begin{corollary}
    If $P$ is a poset, \[L_P=L_{P^*}.\]
\end{corollary}

This corollary suggests that the number of minimal and the number of maximal elements of $P$ cannot be uncovered from $L_P$.

\begin{definition}
    The \textit{linear polynomial} $l_P$ of a poset $P$ is the principal specialization of its linear function $L_P.$
\end{definition}

\begin{example}
    If $C_n$ is the chain on $[n]$, then $C_n=(C_1)^n$. Therefore, $L_{C_n}=(L_{C_1})^n=(M_{(1)})^n$ and consequently $l_P(m)=m^n.$
\end{example}

We mention that the definition of the linear function can be extended to the Hopf algebra of digraphs in accordance with the morphism $P\mapsto D_P$. Namely, if $D=(V, E)$ is a digraph, $\omega\in\Sigma_V$ is a linear extension of $D$ if there do not exist $i<j$ such that $(\omega_j,\omega_i)\in  E$. We define $\zeta(D)$ as the number of linear extensions of $D$. It should be clear that $\zeta$ is multiplicative, since any linear extension $\omega=(\omega_1, \ldots, \omega_n)$ of $D_1D_2$ uniquely splits to a linear extension $(\omega_1, \ldots\omega_{|D_1|})$ of $D_1$ and a linear extension $(\omega_{|D_1|+1}, \ldots, \omega_{n})$ of $D_2$. 

For a relation $R\subseteq V\times V$, its \textit{transitive closure}, denoted as $R_c$, is the smallest relation containing $R$ that is transitive. Clearly, if $D$ has nontrivial directed cycles, then $\zeta(D)=0$. On the other hand, if $D=(V, E)$ is acyclic, then its transitive closure $D_c=(V, E_c)$ can be seen as $D_P$ for some poset $P$. 

\begin{theorem}
    For any acyclic digraph $D$, $\zeta(D)=\zeta(D_c)$.
\end{theorem}

\begin{proof}
    Clearly, $\zeta(D_c)\leq\zeta(D)$ since any linear extension of $D_c$ is a linear extension of $D$. On the other hand, if we suppose that $\omega=(\omega_1, \ldots, \omega_n)$ is a linear extension of $D$ and that it is not a linear extension of $D_c$, then there exist $i<j$ such that $(\omega_j, \omega_i)$ is an edge of $D_c$. Therefore, there is a list $j=i_1, i_2, \ldots, i_k=i$ such that $(\omega_{i_l}, \omega_{i_{l+1}})$ is an edge of $D$ for every $l\in[k-1]$. Since $i=i_k<i_1=j$, there exists at least one $l\in[k-1]$ such that $i_{l+1}<i_l$. The fact that $(\omega_{i_l}, \omega_{i_{l+1}})$ is an edge of $D$ contradicts the fact that $\omega$ is a linear extension of $D$.
\end{proof}

In other words, for the sake of examining properties of the character $\zeta$ and the corresponding symmetric function, it is enough to focus on posets.

\subsection{Expansions of $L_P$}

In this subsection, we are going to explore the coefficients of the expansions of $L_P$ in various natural bases of $Qsym$ and $Sym$. These coefficients provide an abundance of interesting combinatorial connections.

The definition of the linear function together with Equation (\ref{canonical}), immediately yields \begin{equation}\label{kanonikal}
    L_P=\sum_{(P_1, \ldots, P_k)\models P}\zeta(P_1)\cdots\zeta(P_k)M_{(|P_1|, \ldots, |P_k|)}.
\end{equation}In other words, if $\alpha=(\alpha_1, \ldots, \alpha_k)\models n$, $[M_\alpha]L_P$ is actually the number of $P$-listings $\omega$ such that $\alpha\in \mathrm{Comp}(P, \omega)$, or equivalently, $\omega\in\Sigma_P(\alpha)$. 
\begin{theorem}\label{mrazvoj}
    If $P$ is a poset, then \[L_P=\sum_{\alpha\models n}\sum_{\omega\in\Sigma_P(\alpha)}M_{\alpha}=\sum_{\alpha\models n}|\Sigma_P(\alpha)|M_\alpha.\]
\end{theorem}
\begin{example}
    If $T_n$ is the trivial poset on $[n]$, then for $\alpha\models n$, any $\omega\in\Sigma_P$ is also in $\Sigma_P(\alpha)$. Therefore, $L_{T_n}=n!\sum_{\alpha\models n}M_{\alpha}.$
\end{example}

 For a set $P=\{p_1, \ldots, p_n\}$ and $f: P\rightarrow\mathbb{N}$, we write $x_f$ for the product $x_{f(p_1)}\cdots x_{f(p_n)}$. If $P$ is a poset, we define $\Sigma_P(f)$ as the set of all $P-$listings $\omega\in \Sigma_P$ such that $f(\omega_1)\leq f(\omega_2)\leq\cdots\leq f(\omega_n)$ and such that $\omega$ is a linear extension on the monochromatic parts of $f$. The sublisting of $\omega$ with elements from $f^{-1}[i]$ is called the $i$th \textit{monochromatic block} of $\omega$. Note that \[|\Sigma_P(f)|=\prod_{i\in f[P]}\zeta(P|_{f^{-1}[i]}).\]
 Theorem \ref{mrazvoj}  easily yields another interpretation of $L_P$.

 \begin{corollary}
If $P$ is a poset, then \[L_P=\sum_{f:P\rightarrow\mathbb{N}}\sum_{\omega\in \Sigma_P(f)}x_f=\sum_{f: P\rightarrow\mathbb{N}}|\Sigma_P(f)|x_f.\]    
 \end{corollary}

It turns out that the coefficients of the expansion of $L_P$ in the fundamental basis of $QSym$ are closely connected to the Euler character $\chi$ of $\widetilde{R}$.

\begin{theorem} \label{fkoeficijent}
    For any poset $P$ with $n$ elements and for any $I\subseteq [n-1]$, \[[F_I]L_P=\sum_{\omega\in\Sigma_P(I)}\chi_{\mathrm{Comp}(P, \omega)}(I).\]
\end{theorem}

\begin{proof}
We start with the expansion given in Theorem \ref{mrazvoj}:
\[L_P=\sum_{\alpha\models n}\sum_{\omega\in\Sigma_P(\alpha)}M_{\alpha}=\sum_{\omega\in\Sigma_P}\sum_{\alpha\in\mathrm{Comp}(P, \omega)}M_{\alpha}=\]\[=\sum_{\omega\in\Sigma_P}\sum_{\alpha\in\mathrm{Comp}(P, \omega)}\sum_{\beta\leq\alpha}(-1)^{|\mathrm{set}(\beta)|-|\mathrm{set}(\alpha)|}F_{\beta}.\]
The fact that $\beta\leq\alpha\in\mathrm{Comp}(P, \omega)$ implies that $\beta\in\mathrm{Comp}(P, \omega)$ as well. Equivalently, $\omega\in\Sigma_P(\beta)$. By rearranging the sums of the previous expression, we obtain \[ 
L_P=\sum_{\beta\models n}\sum_{\omega\in\Sigma_P(\beta)}\sum_{\substack{\alpha\in\mathrm{Comp}(P, \omega)\\    \beta\leq\alpha}}(-1)^{|\mathrm{set}(\beta)|-|\mathrm{set}(\alpha)|}F_\beta.
\]
If we set $I=\mathrm{set}(\beta)$ and $J=\mathrm{set}(\alpha)$, we get 
\begin{equation}\label{frazvoj}
L_P=\sum_{I\subseteq [n-1]}\sum_{\omega\in\Sigma_P(I)}\sum_{\substack{J\subseteq I\\
                  J\in\mathrm{Comp}(P, \omega) }}(-1)^{|I|-|J|}F_I.
\end{equation}
Therefore,
\[[F_I]L_P=\sum_{\omega\in \Sigma_P(I)}\sum_{\substack{J\subseteq I\\
                  J\in\mathrm{Comp}(P, \omega) }}(-1)^{|I|-|J|}=\sum_{\omega\in\Sigma_P(I)}\chi_{\mathrm{Comp}(P, \omega)}(I).\]\end{proof}

In what follows, we try to find the coefficients of $L_P$ in the power sum basis. We start with a lemma that covers only a few special cases, but it has some interesting consequences. In the proof of this lemma, we will exploit the result given in Equation (\ref{kanonikal}).

\begin{lemma} \label{ppocetak}
    If $P$ is a nonempty poset and $\mathrm{inc}(P)=(P, E)$ its incomparability graph, then 

    a) $[p_{(1, \ldots, 1)}]L_P=1$

    b) $[p_{(2, 1, \ldots, 1)}]L_P=|E|$

\end{lemma}

\begin{proof} 
    
    Let $L_P=\sum c_\lambda p_\lambda$. The only $p_\lambda$ that contains monomial $x_1x_2\cdots x_n$ is $p_{(1, 1, \ldots, 1)}$. More precisely, $x_1x_2\cdots x_n$ appears $n!$ times in $p_{(1, 1, \ldots, 1)}$. Therefore, $n!c_{(1, 1, \ldots, 1)}=[x_1x_2\cdots x_n]L_P=[M_{(1, 1, \ldots, 1)}]L_P=n!$, which proves the first part.

    For the second part, we focus on the term $x_1^2x_2\cdots x_{n-1}$ that appears $\frac{n!}{2}$ times in $p_{(1, 1, \ldots, 1)}$  and $(n-2)!$ times in $p_{(2, 1, \ldots, 1)}$. Therefore, $\frac{n!}{2}+(n-2)!c_{(2, 1, \ldots, 1)}=[x_1^2x_2\cdots x_{n-1}]L_P=[M_{(2, 1, \ldots, 1)}]L_P$. In the partition $(P_1, \ldots, P_{n-1})$ of the set $P$ of type $(2, 1, \ldots, 1)$, the first block can consist of two incomparable elements or of two elements that are comparable in $P$. In the first case, there are two possible linear extensions of $P_1$ and in the second one, there is only one. Therefore, $[M_{(2, 1, \ldots, 1)}]L_P=2|E|(n-2)!+\left(\binom{n}{2}-|E|\right)(n-2)!$ and hence $\frac{n!}{2}+(n-2)!c_{(2, 1, \ldots, 1)}=2|E|(n-2)!+\left(\binom{n}{2}-|E|\right)(n-2)!$. Consequently, $c_{(2, 1, \ldots, 1)}=|E|$.
\end{proof}

\begin{corollary}
    If $P_1$ and $P_2$ are such that $L_{P_1}=L_{P_2}$, then $P_1$ and $P_2$ have the same number of incomparable pairs. Consequently, they have the same number of comparable pairs as well.
\end{corollary}

Another consequence reveals something about the structure of CHA $(\mathcal{P}, \zeta)$.

\begin{corollary}   Let $P$ be a poset.

    a) $P\in S_+(\mathcal{P}, \zeta)\iff P=\emptyset$.
    
    b) $P\in S_-(\mathcal{P}, \zeta)\iff P$ is a chain.
\end{corollary}

\begin{proof}
    If $\Psi: (\mathcal{P}, \zeta)\rightarrow (Sym, \zeta_S)$ is a universal morphism of CHAs, then  $\Psi(S_+(\mathcal{P}, \zeta))\subseteq S_+(Sym, \zeta_S)$ and $\Psi(S_-(\mathcal{P}, \zeta))\subseteq S_-(Sym, \zeta_S)$. We say that partition is \textit{even} (\textit{odd}) if all of its parts are even (odd). $S_+(Sym, \zeta_S)$ is linearly spanned by the set $\{p_\lambda\mid\lambda\text{ is even}\}$, while $S_-(Sym, \zeta_S)$ is linearly spanned by the set $\{p_\lambda\mid\lambda\text{ is odd}\}$ \cite{ABS}. 

     If $P\neq\emptyset$, then $[p_{(1, 1, \ldots, 1)}]L_P=1$ and therefore $L_P=\Psi(P)\notin S_+(Sym, \zeta_S)$. Consequently, $P\notin S_+(\mathcal{P}, \zeta)$ and the converse implication of part a) is trivial. If $P$ is not a chain, then $[p_{(2, 1, \ldots, 1)}]L_P\neq 0$. Hence, $L_P=\Psi(P)\notin S_-(Sym, \zeta_S)$ which implies that $P\notin S_-(\mathcal{P}, \zeta)$. Conversely, let $P=C_n$ be the chain with $n$ elements. Since $\overline{\zeta}(C_1)=\zeta^{-1}(C_1)=-1$, we have that $C_1\in S_-(\mathcal{P}, \zeta)$ and consequently $C_n=(C_1)^n\in S_-(\mathcal{P}, \zeta)$.
\end{proof}

In order to further explore the expansion of $L_P$ in the power sum basis, we need some additional concepts. A poset is $(a+b)-$\textit{free} if it does not contain an induced subposet isomorphic to a disjoint union of two chains with $a$ and $b$ elements. A partition $\{P_1, \ldots, P_k\}\vdash P$ is \textit{stable} if $P_i=P|_{P_i}$ is a trivial poset for every $i\in[k]$. Let $\Pi_P$ denote the collection of all stable partitions of $P$. 

\begin{lemma} \label{slobodan}
    If $P$ is $(2+2)-$free, then the number of linear extensions of $P$ is \[\zeta(P)=\sum_{\{P_1, \ldots, P_k\}\in \Pi_P}\prod_{i=1}^{k}(|P_i|-1)!.\]
\end{lemma}

\begin{proof} Recall that by $D_P(p)$, we denoted the set of all $q\in P$ such that $p\leq_P q$. Let $a=(a_1, \ldots, a_n)$ be one fixed $P-$listing such that $|D_P(a_i)|\geq |D_P(a_{i+1})|$ for every $i\in [n-1]$. Note that $a$ is a linear extension of $P$ since if $p<_Pq$, then $D_P(q)\subsetneq D_P(p)$ and hence $|D_P(q)|<|D_P(p)|$.

Let $\omega=(\omega_1, \ldots, \omega_n)$ be any linear extension. If $\omega_i=a_1$, then the predecessors $\omega_1, \ldots, \omega_{i-1}$ of $\omega_i$ must satisfy the following. Since $a_1$ is minimal in $P$, $\omega_1, \ldots, \omega_{i-1}$ are incomparable with it. Further, $\omega_1, \ldots, \omega_{i-1}$ are also minimal. Namely, every element $p\in P$ that is incomparable with $a_1$ is minimal as well. If not, there would exist $q$ such that $q$ is minimal and $q <_Pp$. If $r>_P a_1$, then either $q$ is incomparable with $r$, or $q<_P r$. The first is not possible since then $\{q, p, a_1, r\}$ would be isomorphic to a disjoint union of two chains of length 2. Therefore, $q<_P r$ for every $r>_P a_1$. Since $p\in D_P(q)\setminus D_P(a_1)$, this implies $|D_P(a)|<|D_P(q)|$, which contradicts the choice of $a_1$. In other words, $P_1=\{\omega_1, \ldots, \omega_{i}\}$ is a trivial subposet of $P$. 

In the sequel, we focus on $P\setminus P_1=\{\omega_{i+1}, \ldots, \omega_n\}$, which is also $(2+2)-$free. We choose $a_l$, the first entry of $a$ that does not appear in $P_1$. If $\omega_{i+j}=a_l$ then in the same manner we prove that $P_2=\{\omega_{i+1}, \ldots \omega_{i+j}\}$ is the set of minimal, mutually incomparable elements of $P\setminus P_1$ and so on. 

Let $\mathcal{L}(P)$ denote the set of all linear extensions of $P$.
Note that the previously described procedure gives rise to a map $ \varphi: \mathcal{L}(P)\rightarrow\Pi_P$ such that $\varphi: \omega\mapsto \{P_1, P_2, \ldots\}$. In the sequel, we want to find $|\varphi^{-1}[\{P_1, \ldots, P_k\}]|$ for every $\{P_1, P_2, \ldots\}\in \Pi_P$.

Let $\{P_1, \ldots, P_k\}\in \Pi_P$. Without loss of generality, we may suppose that $\text{min}\{i\mid a_i\in P_j\}\leq\text{min}\{i\mid a_i\in P_{j+1}\}$ for every $j\in [k-1]$. Therefore, $a_1\in P_1$. According to previous consideration, since $P_1$ is trivial, every element in $P_1$ is minimal. There are $(|P_1|-1)!$ listings of $P_1$ that place $a_1$ on the last place. Similarly, every element of $P_2$ is minimal in $P\setminus P_1$. There are $(|P_2|-1)!$ listings of $P_2$ that place the first entry of $a$ that did not appear in $P_1$ on the last place and so on. Since any element of $P_{s+1}$ is minimal in $P\setminus \bigcup_{j=1}^s P_j$, it is not difficult to see that any concatenation of such listings yield a linear extension of $P$. Moreover, such linear extension maps exactly to $\{P_1, \ldots, P_k\}$ via $\varphi$. Therefore, \[\zeta(P)=|\mathcal{L}(P)|=\sum_{\{P_1, \ldots, P_k\}\in \Pi_P}|\varphi^{-1}[\{P_1, \ldots, P_k\}]|=\sum_{\{P_1, \ldots, P_k\}\in \Pi_P}\prod_{i=1}^{k}(|P_i|-1)!,\]as desired.
 \end{proof}

If $G=(V, E)$ is a graph, let $\mathfrak{S}(G)$ denote the set of all permutations $\sigma$ of the vertex set $V$ such that every cycle of $\sigma$ in the unique cycle decomposition is a clique in $G$. For a permutation $\sigma$, we define $\mathrm{type}(\sigma)$ as a partition whose entries correspond to the lengths of these cycles. 

It is known that the incomparability graph of any poset does not contain an induced subgraph isomorphic to a cycle of length greater than 4. Note that a poset $P$ is $(2+2)-$free if and only if $\mathrm{inc}(P)$ does not contain an induced subgraph isomorphic to the cycle with 4 vertices. For this class of posets, we are able to obtain the following generalization of Lemma \ref{ppocetak}.

\begin{theorem} \label{prazvoj}
    If $P$ is a $(2+2)-$free poset, then \[L_P=\sum_{\sigma\in\mathfrak{S}(\mathrm{inc(P)})}p_{\mathrm{type}(\sigma)}.\]
\end{theorem}
\begin{proof}

We begin by recalling the result from Theorem \ref{mrazvoj} \[L_P=\sum_{f: P\rightarrow\mathbb{N}}|\Sigma_P(f)|x_f.\] For any coloring $f$, monochromatic blocks $P_i=f^{-1}[i]$ for $i\in f[P]$ form a composition of $P$. Every subposet of $P$ is also $(2+2)-$free and hence the cardinality of $\Sigma_P(f)$ can be calculated using Lemma \ref{slobodan}:
\begin{equation}\label{prvo}|\Sigma_P(f)|=\prod_{i\in f[P]}\zeta(P_i)=\prod_{i\in f[P]}\sum_{\{Q_i^1, \ldots, Q_i^{l_i}\}\in\Pi(\mathrm{inc}(P_i))}\prod_{j=1}^{l_i}(|Q_i^{j}|-1)!.\end{equation}
    On the other hand, for any permutation $\sigma$, \[p_{\mathrm{type}(\sigma)}=\sum_{f\sim\sigma}x_f,\]
    where $f\sim\sigma$ denotes that $f$ is a coloring monochromatic on cycles of $\sigma$. Therefore,
    \[\sum_{\sigma\in\mathfrak{S}(\mathrm{inc(P)})}p_{\mathrm{type}(\sigma)}=\sum_{\sigma\in\mathfrak{S}(\mathrm{inc(P)})}\sum_{f\sim\sigma}x_f=\sum_{f}\sum_{\substack{\sigma\in\mathfrak{S}(\mathrm{inc(P)})\\
                  f\sim\sigma}}x_f.\] Every $\sigma$ such that $f\sim\sigma$ permutes the elements within one monochromatic block $P_i$. 
                  For a graph $G=(V, E)$, let $\Pi'(G)$ denote the set of all partitions of $V$ such that every block induces a clique in $G$.
                  
                  The unique cycle decomposition of permutation tells us that $\sigma|_{P_i}$ can be chosen in \[\sum_{\{Q_i^1, \ldots, Q_i^{l_i}\}\in\Pi'(\mathrm{inc}(P_i))}\prod_{j=1}^{l_i}(|Q_i^{j}|-1)!\] ways. 
                  Therefore, $\sigma$ can be chosen in \begin{equation}\label{drugo}
                     \prod_{i\in f[P]}\sum_{\{Q_i^1, \ldots, Q_i^{l_i}\}\in\Pi'(\mathrm{inc}(P_i))}\prod_{j=1}^{l_i}(|Q_i^{j}|-1)! 
                  \end{equation} ways.
                  An induced subgraph of $\mathrm{inc}(P_i)$ is a clique if and only if the adequate induced subposet of $P_i$ is trivial. In other words, $\{Q_i^1, \ldots, Q_i^{l_i}\}\in\Pi'(\mathrm{inc}(P_i))$ if and only if $\{Q_i^1, \ldots, Q_i^{l_i}\}\in\Pi(P_i)$. Therefore, the terms in Equations (\ref{prvo}) and (\ref{drugo}) are the same, which completes the proof.
\end{proof}

\begin{corollary}\label{ppozitivan}
    If $P$ is $(2+2)-$free, $L_P$ is $p-$positive and $p-$integral.
\end{corollary}

This is not true in general. For example, if $P$ is a disjoint union of two copies of $C_2$, an explicit calculation shows that $L_P=p_{1,1,1,1}+4p_{2,1, 1}+2p_{2, 2}-p_4.$ Note that the first two coefficients can be obtained from Lemma \ref{ppocetak}. In addition, the first three coefficients are compatible with the result given in Theorem \ref{prazvoj}, but the last one is not. The fact that $\text{ps}^1(p_\lambda)(m)=m^{l(\lambda)}$ yields the following.

\begin{corollary}\label{polinom}
    If $P$ is $(2+2)-$free, $l_P$ has positive integer coefficients.
\end{corollary}

At this moment, it does not seem to us that $l_P$ satisfies some form of the deletion-contraction property, unlike its relatives - the chromatic polynomial and the Redei-Berge polynomial. Beside being a computational disadvantage, this prevents us from proving that $l_P$ always has integer coefficients.

Recall that the \textit{clique number} of a graph $G$, denoted as $\omega(G)$, is the size of the largest clique in $G$. The corresponding concept for posets in accordance with the assignment $P\mapsto \mathrm{inc}(P)$ is the \textit{incomparability number} $i(P)$, which we define as the size of the largest trivial subposet in $P$. 

\begin{corollary}
    If $P_1$ and $P_2$ are $(2+2)-$free posets and $L_{P_1}=L_{P_2},$ then $i(P_1)=i(P_2)$.
\end{corollary}
\begin{proof}
    This follows immediately from Theorem \ref{prazvoj} since it implies that $i(P)=\omega(\mathrm{inc}(P))=\mathrm{max}\{k\mid [p_{(k, 1, \ldots, 1)}]L_P\neq 0\}.$
\end{proof}

In the sequel, we explore the expansion of $L_P$ in Schur basis for which we need some additional notation. If $P$ is a poset with $n$ elements and $\alpha\models n$, a $P-$\textit{filling} of $\alpha$ is a filling of the diagram of $\alpha$ with elements of $P$ such that every row is a linear extension of appropriate subposet. We say that such $P-$filling has shape $\alpha$. If, in addition, $\alpha\vdash n$ and the filling is strictly increasing along the columns, we say that this filling is a $P-$\textit{tableau}. Note that this condition can be reformulated as follows: \begin{equation}\label{uslov}
    \text{If $a_{i+1, j}$ is defined, then $a_{i, j}$ is also defined and $a_{i, j}<_Pa_{i+1, j}.$}
\end{equation}

\begin{example}\label{2+1}
Let $P$ be the poset whose digraph $D_P$ is shown in the diagram below. Of the three $P-$fillings that are given, only the first one is a $P-$tableau.
    
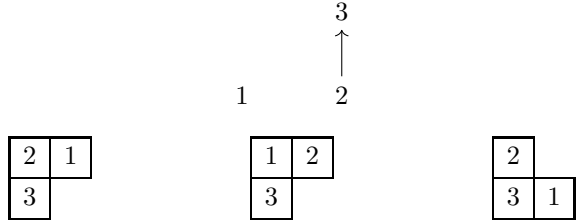
\begin{center}
   \begin{tikzcd}
  & 3           \\
1 & 2 \arrow[u]
\end{tikzcd}
    \end{center}  

    \centering
    \begin{ytableau}
2 & 1 \\
3 
\end{ytableau}\hspace{20mm} 
\begin{ytableau}
1 & 2 \\
3 
\end{ytableau}\hspace{20mm}
\begin{ytableau}
2  \\
3 & 1
\end{ytableau}
\end{example}

The following proof will be a slight modification of the proof from \cite{G} which deals with the $s-$positivity of the chromatic symmetric function.

\begin{theorem} \label{srazvoj}
    If $P$ is a $(2+1)-$free poset, then $[s_{\lambda}]L_P$ is the number of $P-$tableaux of shape $\lambda$.
\end{theorem}

\begin{proof}
  If $\lambda=(\lambda_1, \ldots, \lambda_k)$ and $\pi\in S_k$ is a permutation of $[k]$, denote by $\pi(\lambda)$ the sequence $\{\lambda_{\pi(j)}-\pi(j)+j\}_{j=1}^k$. The definition of the Schur function $s_\lambda$ is equivalent to \[s_\lambda=\sum_{\pi\in S_k}\mathrm{sgn}(\pi)h_{\pi(\lambda)}.\] In this expression, for any integer sequence $\alpha=(\alpha_1, \ldots, \alpha_k)$, we set $h_\alpha=h_{\alpha_1}\cdots h_{\alpha_l}$, where we take $h_r=0$ for $r<0.$ Let $L_P=\sum c_\lambda s_\lambda$. Since $s_\lambda$ form an orthonormal basis for $Sym$, \[c_\lambda=\langle L_P, s_\lambda\rangle=\langle L_P, \sum_{\pi\in S_k}\mathrm{sgn}(\pi)h_{\pi(\lambda)}\rangle=\sum_{\pi\in S_k}\mathrm{sgn}(\pi)\langle L_P, h_{\pi(\lambda)}\rangle.\] Since $\langle m_\lambda, h_\mu\rangle =\delta_{\lambda, \mu}$, we have that $\langle L_P, h_{\pi(\lambda)}\rangle$ is the coefficient of $m_{\pi(\lambda)}$ in $L_P$. It is not difficult to see that this coefficient equals the number of $P-$fillings of shape $\pi(\lambda).$ Let \[A=\{(\pi, T)\mid \pi\in S_k\text{ and $T$ is a $P-$filling of shape $\pi(\lambda)$}\},\]\[B=\{(\pi, T)\in A\mid T \text{ is not a $P-$tableau}\}.\] Then, $c_\lambda=\sum_{(\pi, T)\in A}\mathrm{sgn}(\pi).$ If $T$ is a $P-$tableau of shape $\pi(\lambda)$, then $\pi(\lambda)$ is a partition. Hence, $\pi=Id$ and consequently $\mathrm{sgn}(\pi)=1$. Therefore, it is enough to find an involution $\varphi: B\rightarrow B$ such that if $(\pi', T')=\varphi(\pi, T)$, then $\mathrm{sgn}(\pi')=-\mathrm{sgn}(\pi)$.

  Let $T$ be a filling of shape $\pi(\lambda)$ that is not a tableau. Let $c$ be the smallest integer $j$ such that the condition given in Equation (\ref{uslov}) fails for some $i$ and let $r$ be the largest such $i$. Define $\pi':=\pi\circ (r, r+1)$ and define $T'$ to be a filling of shape $\pi'(\lambda)$ obtained from $T$ by switching the elements in row $r$ after column $c-1$ with those in row $r + 1$ after column $c$. Explicitly, if $T=(a_{i, j})$ and $T'=(b_{i, j})$, then 
  \begin{itemize}
      \item $b_{i, j}=a_{i, j}$ for $(i\neq r \land i\neq r+1)\lor(i=r\land j\leq c-1)\lor (i=r+1\land j\leq c).$
      \item $b_{r, j}=a_{r+1, j+1}$ for $j\geq c$ and $a_{r+1, j+1}$ is defined
      \item $b_{r+1, j}=a_{r, j-1}$ for $j\geq c+1$ and $a_{r, j-1}$ is defined.
  \end{itemize} We define $\varphi (\pi, T)$ to be $(\pi', T')$.
First, we need to show that $(\pi', T')\in B$. Since we only changed the content of rows $r, r+1$, other rows are still linear extensions. 

We first prove that the row $r$ in $T'$ is a linear extension. If $c=1$, this trivially holds. If $c\geq 2$, then for any $k<c$, $a_{r, k}<_P a_{r+1, k}$ due to the fact that $c$ is the smallest integer $j$ for which $a_{i, j}\nless_P a_{i+1, j}$ for some $i$. Since our poset is $(2+1)-$free, $a_{r+1, s}$ for $s>c$ is comparable with at least one of $a_{r, k}$ and $a_{r+1,k}$. But, $a_{r+1, s}\nless_P a_{r+1, k}$ since the row $r+1$ in $T$ is a linear extension. This implies that $a_{r, k}<_P a_{r+1, s}$. Consequently, the row $r$ in $T'$ is a linear extension.

We now focus our attention on the row $r+1$ in $T'$. First suppose that $a_{r+1, c} <_P a_{r, c}$. Since $P$ is $(2+1)-$free and for any $k<c$ $a_{r+1, c}\nless_P a_{r+1, k}$, we have that $a_{r+1, k}<_P a_{r, c}$. Similarly, for any $s\geq c$, $a_{r, s}\nless_P a_{r, c}$ implies that $a_{r+1, k}<_P a_{r, s}$. Therefore, the row $r+1$ is a linear extension. Suppose now that $a_{r+1, c}$ and $a_{r, c}$ are incomparable. If we suppose that the row $r+1$ is not a linear extension in $T'$, then there exists $k\leq c$ and $s\geq c$ such that $a_{r, s}<_P a_{r+1, k}$. Since $a_{r+1, c}\nless_P a_{r+1, k}$, we have that $a_{r, s}<_P a_{r+1, c}$. But, since $a_{r, s}\nless_P a_{r, c}$, we obtain that $a_{r, c}<_P a_{r+1, c}$. Thich is impossible since we supposed that $a_{r, c}$ and $a_{r+1, c}$ are incomparable. Moreover, it contradicts the fact that $T$ does not satisfy the condition given in Equation (\ref{uslov}) for $i=r$ and $j=c$. Hence, in this case, the row $r+1$ is a linear extension as well.
  
Until now, we have proved that $T'$ is a $P-$filling of type $\pi'(\lambda)$. Since $a_{r+1, c+1}\nless_P a_{r+1, c}$, we have that $b_{r, c}\nless_P b_{r+1, c}$. Therefore, the condition from Equation (\ref{uslov}) is not satisfied in $T'$ and hence, $(\pi', T')\in B$. Moreover, it is not difficult to see that in $T'$, $c$ is still the smallest integer $j$ such that (\ref{uslov}) fails for some $i$ and that $r$ is still the largest such $i$. This implies that $\varphi(\pi', T')=(\pi, T)$, which means that $\varphi$ is an involution. The fact that $\mathrm{sgn}(\pi')=-\mathrm{sgn}(\pi)$ completes the proof. 
\end{proof}

\begin{corollary}
    If $P$ is $(2+1)-$free, then $L_P$ is $s-$positive and $s-$integral.
\end{corollary}

We note that even for the smallest digraph that is not $(2+1)-$free, this is not true. Namely, if $P$ is a poset from Example \ref{2+1}, then $L_P=3s_3+2s_{2,1}-s_{1,1,1}$. 

Recall that $h(P)$ denotes the length of the longest chain in $P$. In terms of graph terminology, $h(P)$ is actually $\alpha(\mathrm{inc}(P))$, where $\alpha(G)$ is the \textit{independence number} of $G$, i.e. the size of the largest discrete subgraph of $G$.

\begin{corollary}
    If $P_1$ and $P_2$ are $(2+1)-$free graphs such that $L_{P_1}=L_{P_2}$, then $h(P_1)=h(P_2)$.
\end{corollary}

\begin{proof}
    The proof follows immediately from Theorem \ref{srazvoj} since it implies that $h(P)=\textrm{max}\{l(\lambda)\mid [s_{\lambda}]L_P\neq 0\}$.
\end{proof}

In the upcoming proof, we will need a linear map $\varphi : QSym\rightarrow \mathbb{Q}[t]$ defined on the basis of fundamental functions by \[
 \varphi(F_I) =
  \begin{cases}
  t(t-1)^i & \text{ if } I=\{i+1, \ldots, n-1\}\\
  0 & \text{ otherwise.} 
  \end{cases}
  \]
This function is convenient for detecting elementary functions since $\varphi(e_{\lambda})=t^{l(\lambda)}$ \cite{SR}.
Unfortunately, we were unable to find the combinatorial interpretation of the coefficients of $L_P$ in the elementary basis. However, we managed to deduce the following.

\begin{corollary}
    If $P$ is a poset on $[n]$, $I_i=\{i+1, \ldots, n-1\}$ and $L_P=\sum c_\lambda e_\lambda$, then \[\sum_{l(\lambda)=k}c_\lambda=\sum_{i\geq k-1}\left ((-1)^{i-k+1}\binom{i}{k-1}\sum_{\omega\in\Sigma_P(I_i)}\chi_{\mathrm{Comp}(P, \omega)}(I_i)\right ).\]
\end{corollary}
\begin{proof}
    We apply $\varphi$ to both sides of $L_P=\sum c_\lambda e_\lambda$. The coefficient of $t^k$ in $\varphi(\sum c_\lambda e_\lambda)$ is exactly $\sum_{l(\lambda)=k} c_{\lambda}.$ On the other hand,  Theorem \ref{fkoeficijent} tells us \[\varphi(L_P)=\varphi\left (\sum_{I\subseteq [n-1]}\sum_{\omega\in\Sigma_P(I)}\chi_{\mathrm{Comp}(P, \omega)}(I)F_I\right )=\]\[=\sum_{I=I_i}\sum_{\omega\in\Sigma_P(I)}\chi_{\mathrm{Comp}(P, \omega)}(I)t(t-1)^i.\]
    Since $t(t-1)^i=\sum_{j=0}^i\binom{i}{j}(-1)^{i-j}t^{j+1}$, we have  \[\varphi(L_P)=\sum_{I=I_i}\sum_{\omega\in\Sigma_P(I)}\sum_{j=0}^i\binom{i}{j}\chi_{\mathrm{Comp}(P, \omega)}(I) (-1)^{i-j}t^{j+1}.\] Therefore, by setting $k=j+1$, we obtain that the coefficient of $t^k$ in $\varphi(L_P)$ is 
    \[\sum_{i\geq k-1}\left ((-1)^{i-k+1}\binom{i}{k-1}\sum_{\omega\in\Sigma_P(I_i)}\chi_{\mathrm{Comp}(P, \omega)}(I_i)\right ),\] which completes the proof.
\end{proof}

\begin{example}
 If $P$ is a poset on $[n]$ and if $L_P=\sum c_{\lambda}e_\lambda$, then $c_{(1, \ldots, 1)}=\zeta(P)$. This is always true for functions that are produced by some CHA. However, this is also obtainable from the previous theorem, since $k=l(\lambda)=n$ implies that the inner sum runs through the set $\Sigma_P(\emptyset)$, which is exactly the set of linear extensions of $P$ and since $\chi_{\mathrm{Comp}(P, \omega)}(\emptyset)=1$. On the other hand,\[
 c_{(n)}=\sum_{i=0}^{n-1}\sum_{\omega\in\Sigma_P(I_i)}\chi_{\mathrm{Comp}(P, \omega)}(I_i).
 \]
\end{example}

\subsection{Decomposition of a poset}

In this subsection, we acquire a new way to decompose any poset into a finite formal sum of mountains from $\mathcal{R}$. This decomposition is multiplicative and is motivated by an interesting character on $\mathcal{P}$ that acts on a poset $P$ by
\[\zeta_1(P)=\sum_{\omega\in\mathrm{Rev}(P)}\chi_{\mathrm{Comp}(P, \omega)}\].
\begin{theorem} \label{karakteric}
   Function $\zeta_1$ is a character on $\mathcal{P}$. 
\end{theorem}

\begin{proof}
    If we apply antipode $S$ of $QSym$ to both sides of Equation (\ref{frazvoj}), while noting that $S(F_I)=(-1)^{n}F_{(I^{\mathrm{op}})^c},$ we get
\[S(L_P)=(-1)^n\sum_{I\subseteq [n-1]}\sum_{\omega\in\Sigma_P(I)}\sum_{\substack{J\subseteq I\\
                  J\in\mathrm{Comp}(P, \omega) }}(-1)^{|I|-|J|}F_{(I^{\mathrm{op}})^c}.\]
Changing the summing variable in the first sum to $I^{\mathrm{op}}$, yields \[
S(L_P)=(-1)^n\sum_{I\subseteq [n-1]}\sum_{\omega\in\Sigma_P(I^{\mathrm{op}})}\sum_{\substack{J\subseteq I^{\mathrm{op}}\\
                  J\in\mathrm{Comp}(P, \omega) }}(-1)^{|I|-|J|}F_{I^c}
\]
However, Lemma \ref{lemica} tells us that $\Sigma_P(I^{\mathrm{op}})=\Sigma_{P^*}(I)^{\mathrm{rev}}$ and hence
\[S(L_P)=(-1)^n\sum_{I\subseteq [n-1]}\sum_{\omega\in\Sigma_{P^*}(I)^{\mathrm{rev}}}\sum_{\substack{J\subseteq I^{\mathrm{op}}\\
                  J\in\mathrm{Comp}(P, \omega) }}(-1)^{|I|-|J|}F_{I^c}\]
Now we change the summing variable of the middle sum to $\omega^{\mathrm{rev}}$ and get
\[S(L_P)=(-1)^n\sum_{I\subseteq [n-1]}\sum_{\omega\in\Sigma_{P^*}(I)}\sum_{\substack{J\subseteq I^{\mathrm{op}}\\
                  J\in\mathrm{Comp}(P, \omega^{\mathrm{rev}}) }}(-1)^{|I|-|J|}F_{I^c}.\]
Finally, changing the summing variable of the inner sum to $J^{\mathrm{op}}$ gives
\[S(L_P)=(-1)^n\sum_{I\subseteq [n-1]}\sum_{\omega\in\Sigma_{P^*}(I)}\sum_{\substack{J^{\mathrm{op}}\subseteq I^{\mathrm{op}}\\
                  J^{\mathrm{op}}\in\mathrm{Comp}(P, \omega^{\mathrm{rev}}) }}(-1)^{|I|-|J|}F_{I^c}.\]
Since Lemma \ref{lemica} tells us that $\mathrm{Comp}(P, \omega^{\mathrm{rev}})^{\mathrm{op}}=\mathrm{Comp}(P^*, \omega)$ and since $J\subseteq I\iff J^{\mathrm{op}}\subseteq I^{\mathrm{op}}$, we have that \[S(L_P)=(-1)^n\sum_{I\subseteq [n-1]}\sum_{\omega\in\Sigma_{P^*}(I)}\sum_{\substack{J\subseteq I\\
                  J\in\mathrm{Comp}(P^*, \omega) }}(-1)^{|I|-|J|}F_{I^c}.\]
Consequently, the coefficient of $F_\emptyset$ in $S(L_P)$ is obtained from the previous formula by setting $I=[n-1]$: 

\[[F_{\emptyset}]S(L_P)=(-1)^{n}\sum_{\omega\in\Sigma_{P^*}}\sum_{J\in\mathrm{Comp}(P^*, \omega)}(-1)^{|[n-1]|-|J|}=\]\[(-1)^{n}\sum_{\omega\in\Sigma_{P^*}}\chi_{\mathrm{Comp}(P^*, \omega)}([n-1])=(-1)^{n}\sum_{\omega\in\Sigma_{P^*}}\chi_{\mathrm{Comp}(P^*, \omega)}.\]   
Since $S$ is an antimorphism of algebras and since $[F_\emptyset] (UV)=[F_\emptyset]U[F_\emptyset]V$ for any $U, V\in QSym$, we have that 
$[F_\emptyset]S(L_{P_2P_1})=[F_\emptyset]S(L_{P_2}L_{P_1})=[F_\emptyset](S(L_{P_1})S(L_{P_2}))=[F_\emptyset]S(L_{P_1})[F_\emptyset]S(L_{P_2}).$ Explicitly,  \[\sum_{\omega\in\Sigma_{(P_2P_1)^*}}\chi_{\mathrm{Comp}((P_2P_1)^*, \omega)}=\sum_{\omega\in\Sigma_{P_1^*}}\chi_{\mathrm{Comp}(P_1^*, \omega)}\sum_{\omega\in\Sigma_{P_2^*}}\chi_{\mathrm{Comp(P_2^*, \omega)}}.\] If we rename $P_1^*$ with $P_1$ and $P_2^*$ with $P_2$, since $(P_2P_1)^*=P_1^*P_2^*$, we obtain
\[\sum_{\omega\in\Sigma_{P_1P_2}}\chi_{\mathrm{Comp}(P_1P_2, \omega)}=\sum_{\omega\in\Sigma_{P_1}}\chi_{\mathrm{Comp}(P_1, \omega)}\sum_{\omega\in\Sigma_{P_2}}\chi_{\mathrm{Comp(P_2, \omega)}}.\]
As we have already noted in Lemma \ref{nestajanje}, in expression $\sum_{\omega\in\Sigma_P}\chi_{\mathrm{Comp}(P, \omega)}$, the only terms that survive correspond to $P-$reversing listings and therefore, we have that \[\sum_{\omega\in\mathrm{Rev}(P_1P_2)}\chi_{\mathrm{Comp}(P_1P_2, \omega)}=\sum_{\omega\in\mathrm{Rev}({P_1)}}\chi_{\mathrm{Comp}(P_1, \omega)}\sum_{\omega\in\mathrm{Rev}(P_2)}\chi_{\mathrm{Comp(P_2, \omega)}}.\] Therefore, $\zeta_1$ is multiplicative, which completes the proof.
\end{proof}

\begin{example}
    If $T_n$ is the trivial poset on $[n]$ for $n\geq 2$, then any listing $\omega$ has $\emptyset$ as the only minimal set of border points since any listing is already a linear extension. Hence, $\mathrm{Rev}(T_n)=\emptyset$ and therefore $\zeta_1(P)=0$.
\end{example}

\begin{example}
    If $C_n$ is the chain on $[n]$ induced by the standard order on $\mathbb{N}$, the only $\omega\in\mathrm{Rev}(C_n)$ is $\omega=(n, n-1, \ldots, 1)$. In that case, $\mathrm{Comp}(C_n, \omega)=\{[n-1]\}$ is a one-element poset. If $\omega$ is some other listing, then there exists $i$ such that $\omega_i<\omega_{i+1}$, which means that $i$ never appears in any minimal set of border points in $\mathrm{Comp}(C_n, \omega)$. Therefore, $\zeta_1(C_n)=\chi_{\{[n-1]\}}=1$. This could also be deduced from the fact that $C_n=(C_1)^n$ and the multiplicativity of $\zeta_1.$
\end{example}

Motivated by the connection between the characters $\zeta_1$ on $\mathcal{P}$ and $\chi$ on $\widetilde{\mathcal{R}}$, we are going to focus on a linear map $\Phi: \mathcal{P}\longrightarrow\widetilde{\mathcal{R}}$ that is given by its action on posets \[P\mapsto \sum_{\omega\in\mathrm{Rev}(P)}\mathrm{Comp}(P, \omega).\] This decomposes any poset into a finite sum of mountains. Beside character compatibility, this decomposition has further nice algebraic properties.

\begin{theorem}
    Function $\Phi:\mathcal{P}\rightarrow\mathcal{R}$ defined previously is an algebra morphism such that $\zeta_1(P)=\chi(\Psi(P)).$
\end{theorem}

\begin{proof}
    The definition of $\zeta_1$ immediately implies that $\zeta_1(P)=\chi(\Phi(P))$. In order to prove multiplicativity, we first need to see that for any two posets $P$ and $Q$ and a $PQ-$reversing listing $\omega=(\omega_1, \ldots, \omega_{|P|+|Q|})$, $\omega_1, \ldots, \omega_Q\in Q$ and $\omega_{|Q|+1}, \ldots\omega_{|P|+|Q|}\in P$. If not, there would exist $i$ such that $\omega_i\in P$ and $\omega_{i+1}\in Q$, but then $i$ would not be in any minimal set from $\mathrm{Comp}(PQ, \omega)$. Therefore, any set of border points of $\omega\in\mathrm{Rev}(PQ)$ must contain $|Q|$ since $\omega_{|Q|+1}<_{PQ}\omega_{|Q|}$. It is not difficult to see that for such $\omega$, $\omega'=(\omega_1, \ldots, \omega_{|Q|})\in\mathrm{Rev}(Q)$ and $\omega''=(\omega_{|Q|+1}, \ldots, \omega_{|P|+|Q|})\in\mathrm{Rev}(P)$. Consequently, any listing $\omega\in\mathrm{Rev}(PQ)$ uniquely splits into a pair of listings $(\omega', \omega'')\in\mathrm{Rev}(Q)\times\mathrm{Rev}(P)$. For such $\omega$, $\mathrm{Comp}(PQ, \omega)\cong\mathrm{Comp}(Q, \omega')\times\mathrm{Comp}(P, \omega'')$ by \[\{i_1, \ldots, i_l, |Q|, j_1, \ldots, j_s\}\mapsto \{i_1, \ldots, i_l\}\times\{j_1-|Q|, \ldots, j_s-|Q|\}.\] The fact that for any two posets $M_1$ and $M_2$, $M_1\times M_2\cong M_2\times M_1$ completes the proof since the product on $\widetilde{\mathcal{R}}$ is exactly the Cartesian product.\end{proof}
This theorem could be very useful for detecting posets $P$ such that $\mathrm{Rev}(P)=\emptyset$ since this condition is equivalent to $\Phi(P)=0$. Since any poset $P$ can be written uniquely as a product of irreducible posets and since $\Phi$ is multiplicative, we may focus our attention only on irreducible posets. Any discrete poset $T_n$ with $n$ vertices is irreducible and, for $n\geq 2$, $\Phi(T_n)=0$ since $\mathrm{Rev}(T_n)=\emptyset$. In this way, we obtain the following.

\begin{corollary}
    If $P$ is a poset that contains a discrete poset $T_n$ for $n\geq 2$ as its factor, then $\mathrm{Rev}(P)=\emptyset.$
\end{corollary}

\begin{example}
     Let $P$ be the poset whose digraph $D_P$ is shown in the diagram below.
 
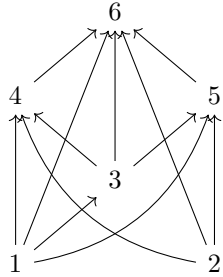
\begin{center}
\begin{tikzcd}
                                                              & 6                                  &                                                   \\
4 \arrow[ru]                                                  &                                    & 5 \arrow[lu]                                      \\
                                                              & 3 \arrow[ru] \arrow[uu] \arrow[lu] &                                                   \\
1 \arrow[ru] \arrow[rruu, bend right] \arrow[ruuu] \arrow[uu] &                                    & 2 \arrow[uu] \arrow[lluu, bend left] \arrow[luuu]
\end{tikzcd}
\end{center}
Note that $P=P_1P_2P_3$, where $P_1=P|_{\{1, 2, 3\}}$, $P_2=P|_{\{4, 5\}}$ and $P_3=P|_{\{6\}}$. Since the second factor is the trivial poset $T_2$, we have that $\mathrm{Rev}(P)=\emptyset$.

\end{example}
However, the implication in the previous corollary is not an equivalence. 
\begin{example}
Let $P$ be the poset whose digraph $D_P$ is shown in the diagram below.
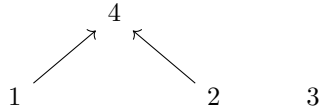
\begin{center}
    \begin{tikzcd}
             & 4 &              &   \\
1 \arrow[ru] &   & 2 \arrow[lu] & 3
\end{tikzcd}
\end{center}
$P$ is an irreducible poset, but $\mathrm{Rev}(P)=\emptyset.$
    
\end{example}

\begin{corollary}\label{paran}
    If $P$ is $(2+2)-$free and $\mathrm{Rev}(P)=\emptyset$, then the number of linear extensions of $P$ is even.
\end{corollary}
\begin{proof}
    The reciprocity formula (\ref{reciprocity}) yields $\text{ps}^1(L_{P^*})(-1)=\text{ps}^1(S(L_{P^*}))(1),$ i.e. $l_{P^*}(-1)=[F_\emptyset](S(L_{P^*})).$ From the proof of Theorem \ref{karakteric}, we see that $[F_\emptyset](S(L_{P^*}))=(-1)^{|P|}\zeta_1(P)$. If $\mathrm{Rev}(P)=\emptyset$, then $\zeta_1(P)=0$ and therefore $l_{P^*}(-1)=0$. If $P$ is $(2+2)-$free, $l_P$ has integer coefficients according to Corollary \ref{polinom}. Consequently, $\zeta(P)=l_P(1)\equiv_2 l_P(-1)=l_{P^*}(-1)=0.$ 
\end{proof}

\section{Further avenues}

In this section, we propose some natural questions and conjectures that could serve as a beacon for the further research of this topic. In the spirit of the conclusion of the last chapter, a question that arises spontaneously is the characterization of the posets $P$ such that $\mathrm{Rev}(P)=\emptyset$, i.e. the posets $P$ that satisfy $\Phi(P)=0$. One step in this direction was made in Corollary \ref{paran}, where we exploited the relationship between the characters $\zeta$ and $\zeta_1$. Is there a way to deepen the understanding of this interconnection? One possible way to acquire that could be finding a deletion-contraction-like recurrence for the linear polynomial $l_P$.

Furthermore, we may ask ourselves if there could exist listings $\omega\in\mathrm{Rev}(P)$ such that $\chi_{\mathrm{Comp}(P, \omega)}=0$. Example \ref{kontraprimer} tells us that there are complete pluckings $\mathcal{A}$ such that $\chi_\mathcal{A}=0$. However, $\mathcal{A}$ from this example cannot be seen as $\mathrm{Comp}(P, \omega)$ for some poset $P$ and $\omega\in\Sigma_P$. This follows from the following consideration.

Let $P$ be a poset, $\omega\in\Sigma_P$ and let $M_1, \ldots, M_k$ be the minimal sets of $\mathrm{Comp}(P, \omega)$. Clearly, if $\omega_l\geq_P\omega_{l+1}$, then $l\in M_j$ for every $j\in [k]$. Conversely, if $\{1, 2\}\subseteq M_i$ for some $i$, then $\omega_1\geq_P\omega_2$ and hence $1\in M_j$ for every $j\in[k]$. Similarly, if $\{n-2, n-1\}\subseteq M_i$, then $\omega_{n-1}\geq_P\omega_n$ and therefore $n-1\in M_j$ for every $j\in[k]$. Finally, for $l\in\{2, \ldots, n-2\}$, if $\{l-1, l, l+1\}\subseteq M_i$ for some $i$, then $\omega_{l}\geq_P\omega_{l+1}$ and consequently $l\in M_j$ for every $j\in [k]$. These conditions are not satisfied in the plucking given in Example \ref{kontraprimer}.

Therefore, in order to answer whether there exists a listing $\omega\in \mathrm{Rev}(P)$ such that $\chi_{\mathrm{Comp}(P, \omega)}=0$, it is not enough to consider the complete plucking sets. We need a deeper understanding of special sets of this kind - the sets that are $\mathrm{Comp}(P, \omega)$ for some $\omega\in\mathrm{Rev}(P)$.

\begin{conjecture}
    If $P$ is a poset and $\omega\in\mathrm{Rev}(P)$, then $\chi_{\mathrm{Comp}(P, \omega)}\neq 0$.
\end{conjecture}

If this conjecture turns out to be false,  we could seek the characterization of pairs $(P, \omega)$ such that $\omega\in\mathrm{Rev}(P)$, but $\chi_{\mathrm{Comp}(P, \omega)}=0$. Let $(f_{0}, \ldots, f_r)$ be the $\widetilde{f}-$vector of $\mathrm{Comp}(P, \omega)$. Since $\chi_{\mathrm{Comp}(P, \omega)}=\sum(-1)^if_i $, we have the following.

\begin{theorem}
    The question of whether $\chi_{\mathrm{Comp}(P, \omega)}=0$ depends only on the $\widetilde{f}-$vector of $\mathrm{Comp}(P, \omega)$.
\end{theorem}

Theorem \ref{prazvoj} gives a combinatorial interpretation of the coefficients of $L_P$ in the power sum basis for $(2+2)-$free posets. Therefore, a natural avenue to pursue is to find the expansion in this basis for posets that are not of this type. The example given after Corollary \ref{ppozitivan} suggests that only certain coefficients differ from what Theorem \ref{prazvoj} claims. Similarly, a generalization of Theorem \ref{srazvoj} would also be valuable.

Finally, we believe that the number $A(P)$ of automorphisms of poset $P$ can be uncovered from its sets of border points. If $\omega\in \Sigma_P$ and $f: P\rightarrow P$ is a bijection, then $f(\omega)=(f(\omega_1), \ldots, f(\omega_n))$ is also a $P$-listing. Moreover, if $f$ is an automorphism of $P$, then $\mathrm{Comp}(P, \omega)=\mathrm{Comp}(P, f(\omega))$. Let $\mathcal{A}_1$, \ldots, $\mathcal{A}_k$ be all mutually different collections of type $\mathrm{Comp}(P, \omega)$ for $\omega\in\Sigma_P$, where $\mathcal{A}_i$ appears for $l_i$ different $P$-listings. Clearly, $l_1+\cdots +l_k=n!$. According to our previous observation, $A(P)\leq l_i$ for every $i$.

\begin{conjecture}
    For any poset $P$, $A(P)=\mathrm{min} \{l_i\}.$
\end{conjecture}




\bibliographystyle{model1a-num-names}
\bibliography{<your-bib-database>}

\begin{thebibliography}{9}

\bibitem{ABS} M. Aguiar, N. Bergeron, F. Sottile, Combinatorial Hopf algebras and generalized Dehn-Sommerville relations, Compositio Mathematica 142 (2006) 1-30.
\bibitem{ABM} J. C. Aval, N. Bergeron, J. Machacek, New invariants for permutations, orders and graphs, Advances in Applied Mathematics 121, (2020), 102080.


\bibitem{CB} C. Berge, Graphs and Hypergraphs, North-Holland Mathematical Library 6, 2nd edition, North-Holland (1976).

\bibitem{BS} N. Bergeron, F. Sotille, Hopf algebras of edge-labeled posets, J. Algebra, 216 (1999), 641–651.








\bibitem{E} R. Ehrenborg, On posets and Hopf algebras, Adv. Math., 119 (1987), 42–99.

\bibitem{G} V. Gasharov, Incomparability graphs of $(3+1)-$free posets are $s-$positive, Discrete Math. 157, 193-197 (1996). 

\bibitem{S} D. Grinberg, R. Stanley, The Redei-Berge symmetric function of a directed graph, \arxiv 2307.05569v1.



\bibitem{GS}  V. Gruji\'c, T. Stojadinovi\'c, The Redei-Berge Hopf algebra of digraphs, \arxiv 2402.07606v2, accepted in Periodica Mathematica Hungarica.




\bibitem{JR} S. Joni, G.C. Rota, Coalgebras and bialgebras in combinatorics, St. Appl. Math., 61
(1979), 93–139.

\bibitem{MS} S. Mitrovi\'c, T. Stojadinovi\'c, Some properties of the Redei-Berge function and related combinatorial Hopf algebras, \arxiv 2407.18608v3.



\bibitem{LR} L. Redei, Ein kombinatorischer Satz, Acta Litteraria Szeged 7 (1934), 39-43.



\bibitem{WRS} W. R. Schmitt, Antipodes and incidence coalgebras, J. Combin. Theory Ser. A, 46 (1987), 264–290.


\bibitem{EC} R. Stanley, Enumerative combinatorics. Vol. 2, Cambridge Univ. Press, Cambridge,
(1999).


\bibitem{SR} R. Stanley, A symmetric function generalization of the chromatic polynomial of a graph, Adv. Math. 111, 166-194 (1995).



\end{thebibliography}




\section{Declarations and statements}

The authors declare that no funds, grants, or other support were received during the preparation of this manuscript. The authors have no relevant financial or non-financial interests to disclose. All authors contributed to the study conception and design. All authors read and approved the final manuscript. 

\section{Data availability}

No data was used for the research described in the article.




\end{document}